\documentclass{elsarticle}

\usepackage[T1]{fontenc}

\usepackage{ifpdf}
\usepackage{float}

\usepackage{breqn}
\makeatletter
\AtBeginDocument{
    \catcode`_=12
    \begingroup\lccode`~=`_
    \lowercase{\endgroup\let~}\sb
    \mathcode`_="8000
    \immediate\write\@auxout{\catcode`_=12 }
    \immediate\write\@auxout{\catcode`^=12 }
}
\makeatother

\usepackage{bm}

\usepackage{mathtools}
\usepackage{xargs}

\usepackage{amssymb, amsmath, amsthm}

\ifpdf
  \usepackage{graphicx}
  \usepackage{hyperref}
\else
  \ifxetex
    \usepackage{graphicx}
    \usepackage{hyperref}
  \else
    \usepackage[dvipdfm]{graphicx}
    \usepackage[dvipdfm]{hyperref}
  \fi
\fi
\graphicspath{{./etc/}}

\usepackage{enumitem}
\usepackage{soul, color}

\usepackage{listings}
\hypersetup{
bookmarksnumbered=true,
colorlinks=true,
pdfstartview={FitH},
linkcolor=blue
}

\DeclareMathOperator\V{V}

\newcommand\Kx{K[\bm{x}]}

\newcommand\Z{\mathbb{Z}}

\newcommand\R{\mathbb{R}}
\newcommand\Proj{\mathbb{P}}

\newcommand{\N}{\mathbb{N}}

\renewcommand\st{\mathrel{\ooalign{$\,\backepsilon$\cr\lower .7pt\hbox{\kern 1pt$-\,$}}}}

\DeclareMathOperator{\Hankel}{\mathcal{H}}

\DeclareMathOperator{\Lt}{Lt}
\DeclareMathOperator{\lt}{lt}

\DeclareMathOperator{\EGH}{EGH}
\DeclareMathOperator{\HF}{HF}

\title{}
\date{}
\theoremstyle{plain}
\newtheorem{theorem}{Theorem}[section]
\newtheorem{lemma}[theorem]{Lemma}
\newtheorem{corollary}[theorem]{Corollary}
\newtheorem{proposition}[theorem]{Proposition}

\theoremstyle{definition}
\newtheorem{definition}[theorem]{Definition}
\newtheorem{remark}[theorem]{Remark}
\newtheorem{example}[theorem]{Example}
\newtheorem{notation}[theorem]{Notation}
\newtheorem{conjecture}[theorem]{Conjecture}
\newtheorem{question}{Question}

\newproof{proofels}{Proof}

\hypersetup{pdfauthor=author}

\begin{document}

\begin{frontmatter}

\title{Strictly positive polynomials in the boundary of the SOS cone}

\author[1]{Santiago Laplagne\corref{cor1}}
\ead{slaplagn@dm.uba.ar}
\author[2]{Marcelo Valdettaro}
\ead{mvaldett@dm.uba.ar}

\cortext[cor1]{Corresponding author}

\affiliation[1]{
organization={Instituto de {Cálculo}, FCEyN, Universidad de Buenos Aires},
addressline={Ciudad Universitaria, Edificio 0 $+$ Infinito},
city={Buenos Aires},
postcode={C1428EGA},
country={Argentina}
}

\affiliation[2]{
organization={Departamento de Matem\'atica, FCEyN, Universidad de Buenos Aires},
addressline={Ciudad Universitaria, Pabell\'on 1},
city={Buenos Aires},
postcode={C1428EGA},
country={Argentina}}

\begin{abstract}
We study the boundary of the cone of real polynomials that can be decomposed as a sum of squares (SOS) of real polynomials. This cone is included in the cone of nonnegative polynomials and both cones share a part of their boundary, which corresponds to polynomials that vanish at at least one point. We focus on the part of the boundary which is not shared, corresponding to strictly positive polynomials.

For the cases of polynomials of degree 6 in 3 variables and degree 4 in 4 variables, this boundary has been completely characterized by G. Blekherman. For the cases of more variables or higher degree, results by G. Blekherman, R. Sinn and M. Velasco and other authors based on general conjectures give bounds for the maximum number of polynomials that can appear in a SOS decomposition and the maximum rank of the matrices in the Gram spectrahedron. Combining theoretical results and computational techniques, we compute examples that allow us to prove the optimality of the bounds for all degrees and number of variables. Additionally, we give examples for the following problems: examples in the boundary of the cone that are the sum of less than $n$ squares and have common complex roots, and examples of polynomials in the boundary with length larger than the expected from the dimension.


\end{abstract}

\end{frontmatter}

\section{Introduction}
\label{introduction}

The decomposition of a real multivariate polynomial as a sum of squares of real polynomials is a central problem in real algebraic geometry, with many theoretical and practical applications.
Any such decomposition gives a certificate of non-negativity and can therefore be used to solve equalities or inequalities.
Of particular importance, the problem of minimizing a polynomial function $f$ can be converted into the problem of finding the maximum $a \in \R$ such that $f - a$ is nonnegative.
Although not every nonnegative polynomial can be decomposed as a sum of squares of real polynomials, fundamental developments by J.B. Lasserre \cite{lasserre} and P. Parrilo \cite{parrilo2003} in the beginning of this century showed that the general minimization problem can be posed as a series of sum of squares decomposition problems. The main interest in this approach is that while the problem of deciding nonnegativity for a general polynomial is computationally hard, the problem of computing a sum of squares decomposition can be solved efficiently by semidefinite programming (a class of convex optimization problems), and many highly optimized software programs are available for this task. This strategy has been used extensively in different areas of mathematics, such as control theory, stability analysis of nonlinear systems, statistics, finance, engineering, robotics and many more.

In 1888, D. Hilbert proved in a famous paper \cite{Hilbert} that every nonnegative polynomial in $n$ variables and even degree $2d$ can be represented as a sum of squares of real polynomials if and only if either (a) $n = 1$ or (b) $d = 2$ or (c) $n = 2$ and $d = 4$. In this work we are interested in better understanding the difference between the set of nonnegative polynomials and the set of polynomials that can be decomposed as sum of squares, which we call SOS polynomials for short. We restrict to the case of homogeneous polynomials, since any nonhomogeneous polynomial can be homogenized preserving nonnegativity and SOS decompositions. For a fixed number of variables $n$ and degree $2d$, the sets of nonnegative polynomials and SOS polynomials define both full dimensional convex cones in the space of homogeneous polynomials.  By the result of Hilbert mentioned above, the SOS cone is in general strictly contained in the cone of nonnegative polynomials. In this work, we study the boundary of the SOS cone that is not part of the boundary of the nonnegative cone. That is, the set of strictly positive polynomials in the boundary of the SOS cone.

Let $H_{n,d}$ be the vector space of homogeneous polynomials in $n$ variables of degree $d$ with coefficients in $\R$. Let $\Sigma_{n,2d} \subset H_{n,2d}$ be the cone of polynomials that can be written as a sum of squares of polynomials in $H_{n,d}$ and let $\partial \Sigma_{n,2d}$ be its boundary. Let $\partial \Sigma^+_{n,2d}$ be the set of strictly positive polynomials in $\partial \Sigma_{n,2d}$.

The condition for $f \in H_{n,2d}$ to be a sum of squares is equivalent to the existence of a real positive semidefinite matrix $Q$ such that $f = v^{T} Q v$, where $v$ is the vector of monomials of degree $d$ in $n$ variables. We will write $Q \succeq 0$ to denote real positive semidefinite matrices. The Gram spectrahedron of $f$ is defined as the set of all $Q \succeq 0$ satisfying the above equality.

In \cite{blekherman}, G. Blekherman studied the sum of squares decompositions of strictly polynomials in $\partial \Sigma_{n,2d}$ for the cases $(n, 2d) = (3,6)$ and $(4,4)$. He obtained the following results.

\begin{theorem}[{\cite[Corollary 1.3]{blekherman}}]\label{blek36} Suppose $f\in\partial\Sigma_{3,6}, f>0$. Then $f$ can be written as a sum of 3 squares and cannot be written as a sum of 2 squares.
\end{theorem}

\begin{theorem}[{\cite[Corollary 1.4]{blekherman}}]\label{blek44} Suppose $f\in\partial\Sigma_{4,4}, f>0$. Then $f$ can be written as a sum of 4 squares and cannot be written as a sum of 3 squares.
\end{theorem}

Moreover, in \cite{capco}, J. Capco and C. Scheiderer prove that the decompositions in the above theorems are essentially unique (that is, unique up to orthogonal transformations, or equivalently, there exists a unique $Q \succeq 0$ such that $f = v^{T}Qv$). For polynomials in more variables or higher degree, G. Blekherman, R.Sinn, and M. Velasco \cite{BSV} provide upper bounds for the number of polynomials in a SOS decomposition of a polynomial $f \in \partial\Sigma_{n,2d}^+$ depending on a conjecture know as the Ottaviani-Paoletti conjecture.

In this work we construct examples showing that the bounds obtained in \cite{BSV} are sharp, and present other examples for related problems. We also show how to obtain the same bounds assuming a conjecture by D. Eisenbud, M. Green, and J. Harris.



\subsection{Questions}

The following questions are motivated by the results in \cite{blekherman} and \cite{BSV}, and we will provide total or partial answers for them.


\begin{question}
\emph{Are the upper bounds given in \cite{BSV} for the number of polynomials in a SOS decomposition of a strictly positive polynomial in the boundary of $\Sigma_{n,2d}$ optimal?} The bounds are known to be optimal for $(n, d) = (3,6)$ and $(4,4)$. We will first show that if the bounds are optimal for a particular $(n_0, d)$, then they are optimal for the same $d$ and any $n \ge n_0$. This, together with the results for $(n, 2d) = (3,6)$ and $(4,4)$ establishes the optimality of the bound for $d=2$ and $d=3$. We will also prove that the bounds are optimal for $d \ge 4$ and $n \ge 3$, providing a general method to construct examples satisfying the bound.
\end{question}

\begin{question}
\emph{What is the minimum number $q$ such that any strictly positive polynomial $f \in \partial\Sigma_{n,d}$ can be decomposed as a sum of squares of $q$ linearly independent polynomials?}
For $(n, 2d) = (3,6)$ and $(4,4)$, by Theorems \ref{blek36} and \ref{blek44}, we know that the polynomial $f$ can always be written as a sum of $n$ squares of linearly independent polynomials. In the general case, by dimension arguments, one can easily get bounds for the minimum number of polynomials in a decomposition. In particular, one can show that $n$ polynomials are in general not enough. One natural question is then if the bounds given by the dimension count can be achieved. We will provide examples where the minimal number of polynomials in any SOS decomposition is larger than the predicted number, giving a negative answer to that question. Example \ref{ex:46sumOf6} implies that $q\ge 6$ for $(n,2d) = (4,6)$ and Example \ref{ex:54sumOf7} proves that $q \ge 7$ for $(n,2d)=(5,4)$. In both cases, these bounds are larger than the bounds given by the dimension count.
\end{question}

\begin{question}
\emph{Can a strictly positive polynomial $f \in \partial\Sigma_{n,d}$ be written as a sum of less than $n$ squares of linearly independent polynomials?} We know that the answer is  negative for the cases $(n,2d)=(3,6)$ and $(4,4)$, again by Theorems \ref{blek36} and \ref{blek44}. Surprisingly, this condition does not hold in the general case. We provide an explicit example of a strictly positive polynomial in $\partial\Sigma_{6,4}$ that can be decomposed as a sum of 5 squares of linearly independent polynomials (see Example \ref{ex:64sumOf5}) and we extend this example to any number of variables $n \ge 6$ (see Example \ref{ex:6nsumOf5}).

Note that if there are less than $n$ homogeneous polynomials in the decomposition, those polynomials will necessarily have common complex roots. So this is another property that cannot be generalized from the cases $(n,2d)=(3,6)$ and $(4,4)$ to the general case. This observation is also discussed in Example \ref{ex:64sumOf5}.
\end{question}

\begin{question}
\emph{How can we verify numerically or symbolically that a polynomial $f$ is in $\partial\Sigma_{n, 2d}$ and that $f$ is strictly positive?}
It is possible to check numerically the condition for $f$ to be in the boundary of the SOS cone using a semidefinite program (SDP) solver like SEDUMI \cite{SEDUMI} and checking whether there is a positive semidefinite matrix in the Gram spectrahedron of $f$ but not a positive definite matrix. Checking this condition symbolically is much more difficult and we do not know any general criterium. However, we prove this condition for Example \ref{ex:reznick} (see Proposition \ref{prop:unique46}), Example \ref{ex:46sumOf5} and Example \ref{ex:54sumOf7}, providing tools that are useful in many cases.
\end{question}

\subsection{Preliminaries}

We start by recalling known facts about the Hilbert function and sequences of parameters.
Let $\Kx = K[x_1,\dots,x_n]$, where $K$ is any field. We note $\Kx_d$ the set of homogeneous polynomials in $\Kx$ of degree $d$.

\begin{definition} We say that $p_1, \dots, p_n \in \Kx_d$ are a sequence of parameters if $\sqrt{\langle p_1, \dots, p_n \rangle} = \langle x_1, \dots, x_n \rangle$.
\end{definition}

The following theorem is a version of Cayley-Bacharach theorem (see \cite[Theorem CB8]{eisenbud96}).

\begin{theorem}
\label{thm:gorenstein}
Suppose that $p_1, \dots, p_n \in \Kx_d$ are a sequence of parameters and let $I$ be the ideal generated by $p_1, \dots, p_n$ in $\Kx$. Then $I$ is a Gorenstein ideal with socle of degree $n(d-1)$.
In particular,  the Hilbert function of $I$ is symmetric ($\HF_k(I) = \HF_{n(d-1)-k}$), $\HF_{n(d-1)}(I) = 1$, and for any homogeneous polynomial $p \not\in I$ of degree smaller than $n(d-1)$, $\HF_{n(d-1)}(I+\langle p \rangle) = 0$. Moreover, the Hilbert function of an ideal generated by a sequence of parameters  is the same for every such sequence.
\end{theorem}

\begin{example}
Let $I = \langle x^3, y^3, z^3 \rangle \subset \R[x,y,z]$. In this case $n = 3$ and $d = 3$, hence $I$ has socle of degree $n(d-1)= 6$. The Hilbert functions for $I$ and $J = I  + \langle x^2 y\rangle$ are shown in Table \ref{table:hilbert}.

\begin{table}[tbhp]
\begin{center}
\begin{tabular}{c|ccccccc}
$d$ & 0 & 1 & 2 & 3 & 4 & 5 & 6  \\ \hline
$\HF_d(I)$ & 1 & 3 & 6 & 7 & 6 & 3 & 1 \\
$\HF_d(J)$ & 1 & 3 & 6 & 6 & 4 & 1 & 0
\end{tabular}
\caption{The Hilbert function for $I = \langle x^3, y^3, z^3\rangle$ and $J = \langle x^3, y^3, z^3, x^2y \rangle \subset \R[x,y,z]$}
\label{table:hilbert}
\end{center}
\end{table}
\end{example}

G. Blekherman uses Theorem \ref{thm:gorenstein} to prove Theorems \ref{blek36} and \ref{blek44}. The crucial condition the author uses to prove both theorems is that for  $(n,2d)=(3,6)$ and $(4,4)$,  $n(d-1) = 2d$, and these are the only two cases for which the equality holds.
In Section \ref{section:Hankel} we review extensions of Theorems \ref{blek36} and \ref{blek44} to arbitrary values of $(n,2d)$.


%
%

\begin{example}
\label{hilbert8}
To see how the situation changes when there are more variables or the polynomials have higher degree, we look at the Hilbert function of $I = \langle x^3, y^3, z^3, w^3 \rangle \subset \R[x,y,z,w]$ (see Table \ref{table:hilbert4}).

\begin{table}[tbhp]
\begin{center}
\begin{tabular}{c|ccccccccc}
$d$ & 0 & 1 & 2 & 3 & 4 & 5 & 6 & 7 & 8  \\ \hline
$\HF_d(I)$ & 1 & 4 & 10 & 16 & 19 & 16 & 10 & 4 & 1
\end{tabular}
\caption{Hilbert function for $I = \langle x^3, y^3, z^3, w^3 \rangle \subset \R[x,y,z,w]$}
\label{table:hilbert4}
\end{center}
\end{table}

By Theorem \ref{thm:gorenstein}, any other ideal generated by a sequence of parameters of the same degrees will have the same Hilbert function. This theorem also says that if we add to $I$ a homogeneous polynomial not in $I$ of degree smaller than or equal to 8, the Hilbert function at $d = 8$ will become 0. In our case,  as we will explain in Sections \ref{section:Hankel} and \ref{section:HF}, we are interested in knowing how many linearly independent polynomials of degree 3 can we add to $I$ without the Hilbert function becoming 0 at $d = 6$. Surprisingly, this number can also be read from the corresponding value of the Hilbert function. We will see in Section \ref{section:Hankel} and  Section \ref{section:HF} that for any ideal $J$ generated by a sequence of parameters in $\R[x,y,z,w]_3$, the value $\HF_6(J) = 10$ implies that whenever we add 10 linearly independent polynomials of degree 3 to $J$, the Hilbert function of the new ideal will become $0$ at $d = 6$. Similarly the value $\HF_7(J)=4$ implies that whenever we add 4 polynomials to $J$, the new Hilbert function will become 0 at $d = 7$. However this property cannot be applied to the whole Hilbert function. In Section \ref{section:HF} (Theorem \ref{thm:HFvsN}) we explicit the values of $d$ for which the number $\HF_d(J)$ gives such information. In the cases $d=6$ and $d=7$ the bound is optimal. That is, if we add one less polynomial, the Hilbert function may not become 0 at that value of $d$.
\end{example}


\section{The Hankel index}
\label{section:Hankel}

In the paper ``Do Sums of Squares Dream of Free Resolutions'', G. Blekherman, R. Sinn and M. Velasco \cite{BSV} study the maximum number of polynomials that can appear in a SOS
decomposition of a strictly positive polynomial in the boundary of the SOS cone in the more general context of totally real varieties.

We restrict in this paper to the space $H_{n,d}$ of homogeneous polynomials of degree $d$ in $\R[x_1, \dots, x_n]$. Specializing their results to our case (that is, considering $X = \nu_d(\Proj^{n-1})$, the $d$th Veronese variety), they call \emph{Hankel index} of $H_{n,d}$, which we note $\Hankel(n,d)$, the smallest number $k \ge 1$ such that any strictly positive polynomial in the boundary of $\Sigma_{n,2d}$ is the sum of at most $\dim(H_{n,d}) - k$ squares of linearly independent polynomials.

The authors observe the following property (see \cite[Theorem 4]{BSV}).


\begin{theorem}
\label{thm:hankel} If $H_{n,2d}$ is included in any ideal generated by $\dim(H_{n,d}) - k$ linearly independent homogeneous polynomials in $H_{n,d}$ containing a sequence of parameters, then any strictly positive polynomial in the border of $\Sigma_{n, 2d}$ is the sum of at most $\dim(H_{n,d}) - k - 1$ polynomials. That is, the Hankel index is at least $k+1$.
\end{theorem}



Moreover, the authors relate the Hankel index to the homological property $N_{2,p}$. In our setting, $H_{n,d}$ is said to satisfy the property $N_{2,p}$ if the first $p-1$ mappings in a minimal free resolution of the $d$th Veronese embedding $\nu(\Proj^{n-1})$ are represented by matrices of linear forms. The largest $p$ for which $H_{n,d}$ satisfies the property $N_{2,p}$ is called the Green-Lazarsfeld index of $H_{n,d}$ (see \cite{BCR}).

\begin{example}
For $H_{3,3}$, we can easily compute in a computer algebra system the Betti table of $\nu_{3}(\Proj^2)$, which is given in Table \ref{table:3rd}. We observe that the first 6 mappings are linear but the 7th is not, hence $H_{3,3}$ satisfies the $N_{2,6}$ property but not the $N_{2,7}$ property. This proves that the Green-Lazarsfeld index of $H_{3,3}$ is 6.

\begin{table}[h!]
\begin{center}
\begin{tabular}{|ccccccccc|}
\hline
  & 0 & 1 & 2 & 3 & 4 & 5 & 6 & 7 \\ \hline
0 & 1 & - & - & - & - & - & - & - \\ \hline
1 & - & 27 & 105 & 189 & 189 & 105 & 27 & - \\ \hline
2 & - & - & - & - & - & - & - & 1 \\ \hline
\end{tabular}
\caption{Graded Betti numbers for the 3rd Veronese embedding of $\Proj^2$}
\label{table:3rd}
\end{center}
\end{table}
\end{example}

In \cite{BSV} the authors observe that if $H_{n,d}$ satisfies the $N_{2,p}$ property, then the Hankel index is at least $p$, this was proved by D. Eisenbud, C. Huneke and B. Ulrich \cite[Corollary 5.2]{EHU}.
Ottaviani and Paoletti \cite{OP} proved that the Green-Lazarsfeld index is 5 for $d=2$ and conjectured that it is $3d-3$ for $d \ge 3$. This conjecture was proved for $d = 3$ and $d = 4$ by Thanh Vu (\cite{ThanhVu1} and \cite{ThanhVu2}) and it is still open for $d \ge 5$. For $H_{3,5}$ (that is, $d = 5$ and $n=3$), the full Betti table is computed in \cite{GrecoMartino} and for $H_{3,6}$ the required entries of the Betti table are computed in \cite{CCDL}, showing that the conjecture also holds in these cases.



We can summarize these results in our setting.
\begin{enumerate}
\item For $d = 2$ and $n \ge 4$, $\Hankel(n,d) \ge 6$,
\item for $d = 3$ and $n \ge 3$, $\Hankel(n,d) \ge 7$,
\item for $d = 4$ and $n \ge 3$, $\Hankel(n,d) \ge 10$,
\item for $d \ge 5$ and $n \ge 3$, it is conjectured that $\Hankel(n,d) \ge (3d - 2)$ (verified computationally for $(n,d)=(3,5)$ and $(n,d)=(3,6)$),
\end{enumerate}
and the number of polynomials that can appear in a SOS decomposition of a polynomial in $\partial \Sigma_{n,2d}^+$ is at most $\dim H_{n,d} - \Hankel(n,d)$.

\begin{example}
A strictly positive polynomial of degree $2d = 8$ in $3$ variables in the boundary of the SOS cone is the sum of at most $15 - 10 = 5$ polynomials, since $\dim H_{3,4} = 15$.
\end{example}

In \cite{BSV}, the authors ask in which cases these bounds are sharp, that is, there exists a polynomial in $\partial \Sigma_{n,2d}^+$ with a SOS decomposition consisting of the maximum possible number of polynomials given by those bounds. The authors relate the above properties to another property. The space $H_{n,d}$ is called $p$-small if for every zero-dimensional ideal $I$ generated by $p$ linearly independent homogeneous polynomials of degree $d$, the points in $\Gamma = V(I)$ are independent (that is, for $\{u_{\gamma}\}_{\gamma \in \Gamma} \subset \R$, $\sum_{\gamma \in \Gamma} u_\gamma q(\gamma) = 0$ for all homogeneous polynomials $q \in H_{n,d}$, if and only if $u_{\gamma} = 0$ for all $\gamma \in \Gamma$).
In \cite[Section 4]{BSV} the authors extend the construction given in \cite[Section 6]{blekherman} for $\Sigma_{3,6}$ and $\Sigma_{4,4}$ to more general varieties. They provide a way to construct a polynomial in $\partial \Sigma_{n,2d}^+$ which can be decomposed as the sum of $\dim(H_{n,d}) - p$ linearly independent polynomials in $H_{n,d}$ when an explicit ideal $I$ showing that $H_{n,d}$ is not $p$-small for a given $p$ is known.
In Section \ref{section:sharpbounds} we will use a refinement of these results to construct explicit examples of strictly positive polynomials in the boundary of the SOS cone $\Sigma_{n,2d}$ that prove that the bounds above are sharp for all values of $n \ge 3$ and $d \ge 4$.

\section{The Hilbert function and the Eisenbud-Green-Harris conjecture}
\label{section:HF}


Let $\Kx = K[x_1,\dots,x_n]$ be the polynomial ring in $n$ variables over a field $K$
and $\Kx_d$ the set of homogeneous polynomials in $\Kx$ of degree $d$, for a fixed $d\ge 1$.
As we have seen in the previous section, to bound the number of polynomials that can appear in a SOS decomposition of a positive polynomial in the border of the SOS cone, we can compute the minimum number $N$ such that for any ideal $I$ generated by $N$ linearly independent homogeneous polynomials of degree $d$ containing a sequence of parameters, the Hilbert function $\HF_{2d}(I) = 0$. We consider in this section the more general problem of computing, for any $k>d$, the minimum number $N(n,d,k)$ such that for any ideal $I$ containing a sequence of parameters generated by $N(n,d,k)$ linearly independent homogenous polynomials of degree $d$, $\HF_k(I) = 0$. Assuming a conjecture by Eisenbud, Green and Harris, we will obtain bounds for $N(n,d,k)$ which for the case $k = 2d$ coincide with the bounds given in the previous section.

\subsection{Leading powers ideals}
\label{subsection:leadingpowers}
We assume first a stronger condition than containing a sequence of parameters, under which the bounds can be proved without depending on conjectures. Most of the results in sections \ref{subsection:leadingpowers} and \ref{subsection:monomialideals} might be known or are direct corollaries of known results, and we include them for lack of concrete references.

\begin{lemma}\label{sqrtI}
Let $I$ be a homogeneous ideal generated by polynomials in $\Kx_d$ such that $\{x_1^d, \dots, x_n^d\} \subset \Lt(I)$, where $\Lt(I)$ is the ideal generated by the leading terms of the polynomials in $I$ for a fixed monomial ordering. Then $\sqrt{I} = \langle x_1, \dots, x_n\rangle$.
\end{lemma}

\begin{proofels}
Let $p \in \Kx_{n(d-1)+1}$ be any homogeneous polynomial of degree $n(d-1)+1$. Let $p_1, \dots, p_n \in I$ be homogeneous polynomials of degree $d$ such that $\lt(p_i) = x_i^d$. The remainder $r$ in the division of $p$ by the polynomials $\{p_1, \dots, p_n\}$ is $0$ or a polynomial of the same degree as $p$ whose leading term is not divisible by any of the leading terms of the polynomials $p_i$, for all $1 \le i \le n$. But, by the pigeonhole principle, for every monomial $m \in \Kx_{n(d-1)+1}$ there exists $i$ such that $\lt(p_i)$ divides $m$, hence $r$ must be $0$ and therefore $p \in I$.
\end{proofels}

We call \emph{leading powers ideal} any ideal satisfying the hypothesis of Lemma \ref{sqrtI}. 
In this section we prove two extensions of Theorem \ref{thm:gorenstein} for the case of leading powers ideals. In Theorem \ref{NHF3} and Corollary \ref{NHF4} we determine for any $k > d$ the minimum number $N=N(n,d,k)$ such that for any leading powers ideal $I$ generated by $N$ linearly independent polynomials of degree $d$, $\HF_k(I) = 0$.

\subsection{Monomial ideals}
\label{subsection:monomialideals}

We will first prove some results for monomial ideals and then extend the results to arbitrary leading powers ideals.

\begin{remark}\label{seqofp} The only sequence of parameters formed by monomials in $\Kx_d$ is $x_1^d,\dots,x_n^d$.
Indeed, let $\{m_1,\dots, m_n\}$ be a monomial sequence of parameters. If $x_1^d$ is not one of those monomials, then $m_i(1,0,\dots,0)=0$, for all $1\le i\le n$ and the monomials in the sequence have a common root.
\end{remark}

The above remark shows that for an ideal $M$ generated by monomials of degree $d$, the following three conditions are equivalent: $M$ contains a sequence of parameters, $M$ is a leading powers ideal, and $M$ contains the monomials $x_1^d,\dots,x_n^d$.


We want to compute, for a monomial ideal $I$ containing $\{x_1^d,\dots, x_n^d\}$, the number of monomials of degree $d$ that the ideal must have in order to assure that $I_k=\Kx_k$, for a given $k>d$, where $I_k$ is the set of polynomials in $I$ of degree $k$. This problem naturally suggests us to study, for a given monomial $m$ of degree $k$ not in $I$, the number of monomials of degree $d$ dividing $m$. Since $\{x_1^d,\dots, x_n^d\}\subset I$, if the monomial $m$ has degree greater than or equal to $d$ in one of the variables, it is obviously in $I$. So  we consider monomials $m$ with degree in each variable bounded by $d-1$, and in particular, we are interested in the monomials with the minimum number of monomial divisors of degree $d$.

We start with the following elementary lemma.

\begin{lemma}
Let $m_1 = x^a y^b$ be a monomial, with $a \ge b$. For a fixed $d \le a + b$, the number of divisors of $m_1$ of degree $d$ is greater than or equal to the number of divisors of $m_2 = x^{a+1}y^{b-1}$.
\end{lemma}
\begin{proofels}
The number of divisors of degree $d$ of $m_1$ is $\min\{a,b,d\} + 1$ and the number of divisors of $m_2$ is $\min\{a+1,b-1,d\} + 1 \le  \min\{a,b,d\} + 1$, because $a \ge b$.
\end{proofels}

We derive the following result.

\begin{corollary}
\label{mnkd} Let $n\ge 1$, $d\ge 2$ and $d\le k\le n(d-1)$. Let $q$ and $r$ be respectively the quotient and the remainder in the division of $k$ by $d-1$. This is, $k=q(d-1)+r$, with $0\le r <d-1$.
The monomial $x_1^{d-1}\cdots x_q ^{d-1}x_{q+1}^r$ has the smallest number of monomial divisors in $\Kx_d$ among all monomials in $\Kx_k$ that are not divisible by any $x_i^d$, $1 \le i \le n$.
\end{corollary}

\begin{proofels}
It follows immediately from the previous lemma. Let $m_1 \in \Kx_k$ be the monomial of degree $k$ with the smallest number of monomials divisors of degree $d$ and not divisible by any $x_i^d$, $1 \le i \le n$. Suppose that there exist $1 \le i, j \le n$ such that the exponents $a_i$, $a_j$ of $x_i$, $x_j$ are positive and smaller than $d-1$, and $a_i \ge a_j$. Set $m_2$ equal to $m_1$ except for the exponents of $x_i$ and $x_j$ which are set as $a_i + 1$ and $a_j -1$ respectively. Then the number of divisors of $m_2$ is smaller than the number of divisors of $m_1$, which is a contradiction.
\end{proofels}

We denote
\begin{equation}
\label{mndk}m(n,d,k) = x_1^{d-1}\cdots x_q ^{d-1}x_{q+1}^r, \end{equation}
 where $q$ and $r$ are the quotient and the remainder in the division of $k$ by $d-1$,
 the monomial of degree $k$ with the smallest number of monomial divisors of degree $d$ among all monomials of degree $k$ not divisible by any element $x_1^{d}, \dots, x_n^d$. Note that if $k > n(d-1)$ there are no monomials satisfying the conditions and if
 $k= n(d-1)$, $q+1> n$ but in this case $r=0$ so we can ignore the last factor.

In the following lemma we compute the number of monomial divisors of degree $d$ of the monomial $m(n,d,k)$, defined in \eqref{mndk}.

\begin{lemma}
\label{lemma:c}
Let $C(n,d,k)$ be the number of divisors of degree $d$ of $m(n,d,k)$. Suppose $k=q(d-1)+r$, with $0\le r <d-1$. Then:
\begin{equation}\label{Cndk}
C(n,d,k)=\binom{q+d}{d}-\binom{q+d-r-1}{q}-q.\end{equation}
\end{lemma}

\begin{proofels} Assume first that $k<n(d-1)$, so $q+1\le n$, and assume also that $r>0$. In this case we can just compute the number of monomials in $q+1$ variables and degree $d$ dividing $m(n,d,k)$. The total number of monomials of degree $d$ in $q+1$ variables is $\binom{q+d}{d}$. In this set, the monomials not dividing $m(n,d,k)$ can be classified in two disjoint cases:
\begin{itemize}[leftmargin=*]
\item Those having degree greater than $r$ in the last variable. There are $\binom{q+d-r-1}{q}$ of them.
\item Those having degree $d$ in one of the first $q$ variables. There are $q$ of them.
\end{itemize}

This leads to the desired formula.
It is easy to see that the same expression holds for the case $r=0$ and also for the case $k=n(d-1)$.
\end{proofels}

Note that $m(n,d,k)$ and $C(n,d,k)$ do not depend directly on $n$, but the value of $n$ bounds the possible values of $k$ in the definition, so we include $n$ in the notation for the sake of clearness.

We introduce now the following notation.

\begin{notation} Set $n\ge 1$, $d\ge 2$ and $d\le k\le n(d-1)$. Suppose $k=q(d-1)+r$, with $0\le r <d-1$. We denote:
\begin{align}
 \label{Nndk} N(n,d,k)&=\dim(H_{n,d})-C(n,d,k)+1\\
 &=\binom{n+d-1}{d}-C(n,d,k)+1\\
\nonumber & =  \binom{n+d-1}{d}-\binom{q+d}{d}+\binom{q+d-r-1}{q}+q+1.\end{align}
\end{notation}

Observe that, since $C(n,d,k)$ is the number of monomials of degree $d$ dividing $m(n,d,k)$, $N(n,d,k)-1$ is the number of monomials of degree $d$ \emph{not dividing} $m(n,d,k)$.


Now we relate the number $N(n,d,k)$ to monomial ideals containing a sequence of parameters (recall Remark \ref{seqofp}).
\begin{proposition}\label{NHF}
Set $n\ge 1$, $d\ge2$, $d\le k \le n(d-1)$, and set $N=N(n,d,k)$ as in \eqref{Nndk}. Let $W = \{x_1^d,\dots,x_n^d\}$ and suppose that $\{m_{n+1},\dots,m_N\}$ is any set of monomials in $\Kx_d\setminus W$. Then the ideal $J =\langle x_1^d,\dots,x_n^d,m_{n+1},\dots,m_N\rangle$ satisfies $J_k=\Kx_k$, where $J_k$ is the set of polynomials in $J$ of degree $k$.

Moreover, the number $N$ is optimal. This is, there exists a set of monomials $\{m_{n+1},\dots,m_{N-1}\}$
 in $\Kx_d\setminus W$, such that for $J=\langle x_1^d,\dots,x_n^d,m_1,\dots,m_{N-1}\rangle$, $J_k \subsetneq \Kx_k$.
\end{proposition}

\begin{proofels}
We will prove the statement by showing that for every monomial $m\in \Kx_k$, we have $m\in J$. If $m=x_1^{i_1}\cdots x_n^{i_n}$, with $i_j\ge d$, for some $1 \le j \le n$, then $m\in J$. So it suffices to consider the monomials $=x_1^{i_1}\cdots x_n^{i_n}$ of degree $k$ such that $i_j<d$, for all $1\le j\le n$. By Corollary \ref{mnkd}, the monomial of this type with the minimum number of divisors is $m(n,d,k)=x_1^{d-1}\cdots x_q ^{d-1}x_{q+1}^r$, and the number of divisors is  $C(n,d,k)$, defined in \eqref{Cndk}. Therefore, if $J$ has at least $N(n,d,k)=\binom{n+d-1}{d}-C(n,d,k)+1$ monomials in degree $d$, necessarily one of them must divide $m(n,d,k)$. Since all other monomials in degree $k$ have more or equal number of divisors, the same is true for all of them and we are done.

The number $N$ is indeed optimal, since we can consider the set of monomials $\{x_1^d,\dots,x_n^d,m_{n+1},\dots,m_{N-1}\}$  in $\Kx_d$ not dividing the monomial $m(n,d,k)$.
Hence $m(n,d,k)\notin J=\langle x_1^d,\dots,x_n^d,m_{n+1},\dots,m_{N-1}\rangle$ and $J_k\neq \Kx_k$.
\end{proofels}

\begin{remark}
Proposition \ref{NHF} and the following results do not cover the case $d=1$ since the numbers $C(n,1,k)$ and $N(n,1,k)$ are not defined. But the case $d=1$ implies $\{x_1, \dots, x_n\} \subset I$, and all claims become trivial.
\end{remark}

We can read Proposition \ref{NHF} in terms of the Hilbert function of a monomial ideal as follows, using the complementary relation between $N(n,d,k)$ and $C(n,d,k)$.

\begin{corollary}\label{NHF2} Set $n\ge 1$, $d\ge 2$, $d\le k \le n(d-1)$, and consider $C(n,d,k)$, defined in \eqref{Cndk}. Let $J$ be an ideal in $\Kx$ generated by monomials in $\Kx_d$ and containing $\{x_1^d,\dots,x_n^d\}$. If $\HF_d(J)=C(n,d,k)-1$, then $\HF_k(J)=0.$

Moreover, the number $C(n,d,k)-1$ is optimal. This is, there exists an ideal $J$ in $\Kx$ generated by monomials in $\Kx_d$ and containing $\{x_1^d,\dots,x_n^d\}$, such that $\HF_d(J)=C(n,d,k)$ and $\HF_k(J)>0.$
\end{corollary}

\begin{proofels} This is just a reformulation of Proposition \ref{NHF}, since by definition:
\begin{align*}
\dim_K(J_d)=N(n,d,k) \iff \HF_d(J)&=\binom{n+d-1}{d}-N(n,d,k) \\
&= C(n,d,k)-1,
\end{align*}
and
\[J_k=\Kx_k \iff \HF_k(J)=0.\]

In the same way, we see that $C(n,d,k)-1$ is optimal.
\end{proofels}

Now we extend Proposition \ref{NHF} and Corollary \ref{NHF2} to homogeneous leading powers ideals, that is, ideals generated by polynomials in $\Kx_d$, not necessarily monomials, and such that $\{x_1^d,\dots,x_n^d\} \subset \Lt(I)$. This condition is more general than the condition $\{x_1^d,\dots,x_n^d\} \subset I$, and it implies that $I$ contains a sequence of parameters (the sequence of polynomials corresponding to those leading terms). However, for an ideal $I$ generated by polynomials in $\Kx_d$, the assumption that $I$ contains a sequence of parameters does not in general imply that $I$ is a leading powers ideal.

\begin{theorem}\label{NHF3} Set $n\ge 1$, $d\ge 2$, $d\le k \le n(d-1)$, and set $N=N(n,d,k)$ as in \eqref{Nndk}. Let $I$ be a homogeneous leading powers ideal in $\Kx$ generated by $N$ linearly independent polynomials in $\Kx_d$. Then $I_k=\Kx_k$.

Moreover, the number $N$ is optimal. This is, there exists a homogeneous leading powers ideal $I  \subset \Kx$ generated by $N-1$ linearly independent polynomials in $\Kx_d$ such that $I_k\subsetneq \Kx_k$.
\end{theorem}

\begin{proofels} We write $I=\langle p_1,\dots,p_N\rangle$. By Gaussian elimination, we can assume that the leading monomials $m_i$ of the polynomials $p_i$ are different, and by hypothesis we can assume that $m_i =x_i^d$, for all $1\le i\le n$. Let $J = \langle m_1, \dots, m_N \rangle$. By Proposition \ref{NHF}, $J_k=\Kx_k$. In particular, each monomial in $\Kx_k$ is divisible by some monomial $m_i$, with $1 \le i \le N$.

Take now a polynomial $q\in \Kx_k$. Suppose that the leading monomial of $q$ is divisible by a monomial $m_{i_0}$. The remainder in the division of $q$ by $p_{i_0}$ is a polynomial $r$ in $ \Kx_k$ with leading monomial smaller than the leading monomial of $q$. We can repeat this process starting with $r$. We obtain in this way a sequence of polynomials in $ \Kx_k$, with strictly decreasing leading monomials, which are the remainders of the successive divisions. This process ends necessarily in a division with remainder $0$ because $J_k=\Kx_k$, and this shows that $q\in I$.

The same example as in the proof of Proposition \ref{NHF} works here to show that $N$ is optimal.
\end{proofels}

As for Proposition \ref{NHF}, we can restate Theorem \ref{NHF3} in terms of the Hilbert function of a leading powers ideal.

\begin{corollary}\label{NHF4} Set $n\ge 2$, $d\ge 2$, $d\le k \le n(d-1)$, and consider $C(n,d,k)$, defined in \eqref{Cndk}. Let $I$ be a homogeneous leading powers ideal in $\Kx$ generated by polynomials of degree $d$. If $\HF_d(I)=C(n,d,k)-1$, then $\HF_k(I)=0.$
\end{corollary}

Moreover, the number $C(n,d,k)-1$ is optimal. This is, there exists a homogeneous leading powers ideal $I$ in $\Kx$, generated by polynomials of degree $d$, satisfying $\HF_d(I)=C(n,d,k)$ and $\HF_k(I)>0.$

\begin{remark}
\label{remark:2d}
When $k = 2d$ as in the previous section,
$$
m(n,d,2d) = \begin{cases}
x_1 x_2 x_3 x_4 & \text{ if  $d = 2$ and $n \ge 4$}, \\
x_1^{d-1} x_2^{d-1} x_3^2 & \text{ if  $d \ge 3$}.
\end{cases}
$$
Counting the number of divisors of degree $d$ or applying Lemma \ref{lemma:c} we obtain $C(n,d,2d) = \binom{4}{2} = 6$ for $d = 2$ and $C(n,d,2d) = (d-1) + d + (d-1) = 3d-2$ for $d \ge 3$. We obtain that $C(n,d,2d)$ is equal to the conjectured values of the Hankel index $\Hankel(n,d)$ detailed in Section \ref{section:Hankel}.
\end{remark}


\begin{example}\label{N436} Take $K[x,y,z,w]$, so $n=4$, and let $d=3$ and $k=6$. In this case we have $q=3$ and $r=0$, so $m(4,3,6)=x^2 y^2 z^2$. The set of monomials in $K[x,y,z,w]_3$ not dividing $m(4,3,6)$ is equal to $C_0\cup C_1$, where $C_0=\{x^3,y^3,z^3\}$, and $C_1$ is the set of all monomials of the form $m\cdot w$, where $m$ is any monomial of degree 2. We have then:
\begin{equation}\label{N436eq}
N(4,3,6)-1=|C_0| + |C_1|= 3+\binom{5}{2}=13,\end{equation}
and thus
\[C(4,3,6)=\binom{6}{3}-N(4,3,6)+1=7.\]

This means that, for a homogeneous leading powers ideal $I$ generated by polynomials in $K[x,y,z,w]_3$: if $I$ is generated by $14$ linearly independent polynomials of degree $3$, by Theorem \ref{NHF3}, necessarily $I_6=K[x,y,z,w]_6$.
In other terms: if $\HF_3(I)=C(4,3,6)-1=6$, by Corollary \ref{NHF4}, we have $\HF_6(I)=0$.
\end{example}

\subsection{The Eisenbud-Green-Harris conjecture}
\label{subsection:EGH}

In Corollary \ref{NHF4}, we showed that if $I$ is a homogeneous leading powers ideal in $\Kx$ generated by polynomials of degree $d$, and $\HF_d(I)=C(n,d,k)-1$, then $\HF_k(I)=0$. In this section we show that if we assume a conjecture by Eisenbud, Green and Harris \cite{EGH93}, we can obtain the same bound for any ideal $I$ that contains a sequence of parameters, extending the results from Section \ref{section:Hankel} to other terms of the Hilbert function of $I$.
We recall first this conjecture.


\begin{conjecture}[EGH] \label{conj:EGH} Let $I$ be a homogeneous ideal in $\R[x_1,\dots,x_n]$ containing a sequence of parameters $p_1, \dots, p_n$ of degrees $\deg(p_i) = a_i$. Then there exists an ideal $J$ containing $\{x_i^{a_i}: 1 \le i \le n\}$ such that $I$ has the same Hilbert function as $J$.
\end{conjecture}

If we assume this conjecture to be true, we can study general homogeneous ideals containing a sequence of parameters by restricting to ideals containing $\{x_i^{a_i}: 1 \le i \le n\}$, where the same results can be easier to prove. The conjecture can be seen as a generalization of Theorem \ref{thm:gorenstein} (Cayley-Bacharach theorem). We are interested in the case $a_i=d$, for all $1\le i\le n$, as in Section \ref{subsection:leadingpowers}, so it will be sufficient to consider the following special case of the EGH conjecture.

\begin{notation}
We say that $\EGH(n,d)$ is true if for any homogeneous ideal $I$ generated by polynomials in $H_{n,d}$ containing a sequence of parameters in degree $d$, there exists an ideal $J$ containing $\{x_1^d, \dots, x_n^d\}$ with the same Hilbert function as $I$.
\end{notation}

Note that in our notation $\EGH(n,d)$ does not cover all cases of the EGH conjecture, since we are assuming that all the polynomials in the system of generators have the same degree.

Now we use the EGH conjecture to generalize the Ottaviani-Paoletti conjecture.

\begin{proposition}
\label{prop:nd}
For $n\ge 2$, $d\ge 2$, $d\le k \le n(d-1)$, let $I \subset K[x_1, \dots, x_n]$ be an ideal generated by a set of $ C(n,d,k) - 1$ linearly independent homogeneous polynomials of degree $d$ containing a sequence of parameters.
If $\EGH(n,d)$ is true, then $\HF_k(I) = 0$.
\end{proposition}

\begin{proofels}
By $\EGH(n,d)$, there exists an ideal $J$ containing $\{x_1^d, \dots, x_n^d\}$ with the same Hilbert function as $I$.
The result now follows from Corollary \ref{NHF4}.
\end{proofels}


By definition, $N(n,d,k) = \dim(H_{n,d}) - C(n,d,k)+1$ and we have seen in Remark \ref{remark:2d} that for $k = 2d$, $C(n,d,2d)$ equals the conjectured value of the Hankel index. Together with Theorem \ref{thm:hankel}, as a corollary of Proposition \ref{prop:nd}, we reobtain the bounds for the number of polynomials in a SOS decomposition stated in Section 2.

\begin{theorem}\label{Nsquares}
Let $f \in \partial \Sigma_{n,2d}^+$. Then, assuming $\EGH(n,d)$ is true, any decomposition of $f$ as a sum of squares of linearly independent polynomials contains at most $N(n,d,2d) - 1$ polynomials.
\end{theorem}


%
%

\subsection{The Hilbert function and \texorpdfstring{$N(n, d, k)$}{N(n, d, k)}}
We show now some examples to illustrate the above results and interpret the numbers $C(n,d,k)$ and $N(n,d,k)$.

\begin{example}\label{N436bis} As in Example \ref{N436}, consider $K[x,y,z,w]$, so $n=4$, and let $d=3$ and $k=6$. Let $J=\langle x^3,y^3,z^3,w^3\rangle$.
This example will allow us to derive a surprising relation between the number $N(4,3,6)$ and the Hilbert function of $J$, besides the one in Proposition \ref{NHF} that stands in general.
The Hilbert function for $J$ is shown in Table \ref{table:hilbert46}.

\begin{table}[ht]
\begin{center}
\begin{tabular}{c|ccccccccc}
$d$ & 0 & 1 & 2 & 3 & 4 & 5 & 6 & 7 & 8  \\ \hline
$\HF_d(J)$ & 1 & 4 & 10 & 16 & 19 & 16 & 10 & 4 & 1
\end{tabular}
\caption{Hilbert function for $J =  \langle x^3,y^3,z^3,w^3\rangle  \subset \R[x,y,z,w]$ }
\label{table:hilbert46}
\end{center}
\end{table}

The cardinal of the set $C_1$ in Example \ref{N436} is by definition equal to $\HF_2(J)$, since $J$ contains no polynomials in degree smaller than 3. By Theorem \ref{thm:gorenstein} we have $\HF_2(J)=\HF_6(J)$, so we can rewrite equation \eqref{N436eq} in Example \ref{N436} in the following terms:
\begin{equation}\label{HN436}
N(4,3,6)=\binom{5}{2}+4=\HF_6(J)+4.\end{equation}

By the Cayley-Bacharach theorems, we already knew that $\HF_8(J) = 1$ and that if we add one homogenous polynomial $f$ of degree 4 to $J$, then $\HF_8(J + \langle f \rangle) = 0$. This example shows that $\HF_6(J) = 10$ implies that if we add 10 polynomials of degree 4 to $J$, then the resulting ideal $J'$ satisfies $\HF_6(J') = 0$. It also holds that if we add $\HF_7(J) = 4$ polynomials to $J$, the resulting ideal $J''$ satisfies $\HF_7(J'') = 0$. That is, the numbers $N(4,3,s)$ for $6 \le s \le 8$ can be read directly from the Hilbert series of $J$.
\end{example}

The relation found in the last example is a particular case of the following result.

\begin{proposition}
\label{HFvsN} Let $n\ge 1$ and $d\ge 2$.  Let $p_1, \dots, p_n$ be a sequence of parameters in degree $d$ and $J = \langle p_1, \dots, p_n \rangle$. Denote $m_d=\max\{n(d-1)-d,d\}$. Let $k$ be such $m_d < k \le  n(d-1)$. Then $N(n,d,k) = \HF_k(J)+n$.
\end{proposition}

\begin{proofels} Set $l=n(d-1)-k$. Since $l+k=n(d-1)$, by Theorem \ref{thm:gorenstein}, we have: $\HF_k(J)=\HF_l(J)$. By definition of $m_d$, we have $l<d$, thus $\HF_l(J)$ is the number of monomials in $\Kx_l$.

Now $N(n,d,k)$ is the number of monomials of degree $d$ not diving $m(n,d,k)$, which under our assumptions is of the form $m(n,d,k) = x_1^{d-1} x_2^{d-1} \dots x_{n-1}^{d-1} x_n^r$, for $r = k - (n-1)(d-1)$.
As in Example \ref{N436}, the set of monomials of degree $d$ not dividing $m(n,d,k)$ is equal to $C_0 \cup C_1$, where $C_0 = \{x_1^d, \dots, x_n^d\}$ and $C_1$ is the set of all monomials of the form $m \cdot x_n^r$, for $m$ any monomial of degree $d-r = l$, hence $|C_0|=n$ and $|C_1| = \HF_l(J)$, which proves the claim.
\end{proofels}

Note that $N(n,d,k) = \HF_k(J)+n$ holds also trivially for $k = d$ by definition of the Hilbert function.

Using Theorem \ref{NHF3}, we deduce from Proposition \ref{HFvsN} the following result.

\begin{theorem}\label{thm:HFvsN}
Let $n \ge 1$, $d\ge 2$, let $p_1, \dots, p_n$ be a sequence of parameters in $\Kx_d$ and let $J =  \langle p_1, \dots, p_n \rangle$.
Let $k > n(d-1) - d$. Assuming the Eisenbud-Green-Harris conjecture, if we add $\HF_k(J)$ linearly independent polynomials of degree $d$ to $J$, the resulting ideal will always contain all $\Kx_k$.
\end{theorem}

Theorem \ref{thm:HFvsN} generalizes Theorem \ref{thm:gorenstein} in the sense of how many polynomials we need to add to the ideal $J$ in order to obtain the value 0 in a specific term of the Hilbert function. See also Example \ref{hilbert8}.

\section{Sharp bounds for the maximum number of polynomials in a SOS decomposition}
\label{section:sharpbounds}

We have seen that, if Ottaviani-Paoletti conjecture or the Eisenbud-Green-Harris conjecture are true, the maximum number of polynomials that can appear in a sum of squares decomposition of a polynomial on the positive boundary of $\Sigma_{n,2d}$ is
$$
\begin{cases}
\dim(H_{n,2d}) - 6 & \text{ for } d = 2 \text{ and } n \ge 4, \\
\dim(H_{n,2d}) - (3d - 2) & \text{ for } d \ge 3 \text{ and } n \ge 3. \\
\end{cases}
$$
By \cite{blekherman}, these bounds are sharp in the cases $(n,d) = (4,2)$ and $(n,d) = (3,3)$.
In this section we establish the sharpness of these bounds for general $d$ and $n$.

We first show that if the bound is sharp for a given $d \ge 2$ and some $n_0 \in \N$, then it is sharp for the same $d$ and any $n \ge n_0$. This is a direct consequence of \cite[Theorem 21]{BSV}.
Recall that the dual cone $K^*$ of a convex cone $K$  in a real vector space $V$ is the set of all linear functionals in the dual space $V^*$ that are nonnegative on $K$: $K^*=\{\ell\in V^*\,:\,\ell(x)\ge 0,\, \forall x\in K\}$. An element $v \in K \subset V$ is in the boundary of $K$ if and only if there exists $\ell \in K^*$ such that $\ell(v) = 0$.

\begin{proposition}
\label{prop:sharp}
For $d \ge 2$ and $n_0 \in \N$, let $f = p_1^2 + \dots + p_s^2$ be a strictly positive polynomial in the boundary of $\Sigma_{n_0,2d}$. For $n > n_0$, let $m_1, \dots, m_r$ be all the monomials in $H_{n,d} \smallsetminus H_{n_0,d}$, that is, all the monomials divisible by some $x_i$, $n_0 + 1 \le i \le n$. Then $g = f + m_1^2 + \dots + m_r^2$ is a strictly positive polynomial in the boundary of $\Sigma_{n,2d}$. If $s = \dim H_{n_0, d} - k$ then $s+ r = \dim H_{n,d} - k$.
\end{proposition}

\begin{proofels}
It is clear that $g$ is a strictly positive polynomial, since $f$ is strictly positive.
To prove that $g \in \partial\Sigma_{n,2d}$, take $\ell \in \Sigma_{n_0,2d}^*$ such that $\ell(f) = 0$. A form $\ell$ with this property always exists since $f \in \partial\Sigma_{n,2d}$. Extend $\ell: H_{n_0, 2d} \rightarrow \R$ to $\tilde \ell: H_{n, 2d}\rightarrow \R$ by setting $\tilde \ell(m(x_1, \dots, x_n)) = \ell(m(x_1, \dots, x_{n_0}, 0, \dots, 0))$ for $m \in H_{n,2d}$. Then $\tilde \ell \in \Sigma_{n,2d}^*$ (since $\ell \in \Sigma_{n_0,2d}^*$) and $\tilde \ell(g) = 0$, hence $g \in \partial\Sigma_{n,2d}$.
\end{proofels}

We obtain the following immediate corollary.

\begin{corollary}
\label{coro:n0n}
Let $d \ge 2$ and $n_0 \in \N$. If $f \in H_{n_0, 2d}$ is a strictly positive polynomial in the boundary of $\Sigma_{n_0,2d}$ and $f$ is a sum of squares of $s = \dim H_{n_0, d} - k$ linearly independent polynomials, for some $k \ge 0$, then for any $n > n_0$ there exists a strictly positive polynomial $g$ in the boundary of $\Sigma_{n,2d}$ that is the sum of squares of $\dim H_{n,d} - k$ linearly independent polynomials.
\end{corollary}

Based on these results and known examples, we get the following sharp bounds.

\subsection{Sharp bounds for \texorpdfstring{$d = 2$}{d = 2}}

\begin{corollary}
\label{coro:d2}
  The maximum number of linearly independent polynomials in a SOS decomposition of a strictly positive polynomial in the boundary of $\Sigma_{n,4}$, $n \ge 4$, is $\dim(H_{n,2}) - 6$ and this bound is sharp for all $n \ge 4$.
\end{corollary}
\begin{proofels}
For $n = 4$, the dimension of $H_{4,2}$ is 10. G. Blekherman \cite{blekherman} provides a direct way to construct examples of strictly positive polynomials in $\partial \Sigma_{4,4}$ which are the sum of squares of 4 linearly independent polynomials (see also \cite{capco} for a concrete example), hence the bound is sharp in this case. By Proposition \ref{prop:sharp}, it is then sharp for $d = 2$ and any $n \ge 4$.
\end{proofels}

\begin{remark}
In \cite{scheiderer2017}, C. Scheiderer uses that every strictly positive polynomial in the boundary of $\Sigma_{4,4}$ is the sum of 4 squares to prove that the Pythagoras number of $H_{4,4}$ is 5 (that is, every SOS polynomial in $H_{4,4}$ can be decomposed as the sum of 5 squares). For $d = 2$ and $n=5$, $\dim(H_{5,2}) = 15$, hence by Corollary \ref{coro:d2} any strictly positive polynomial in the boundary of $\Sigma_{5,4}$ is the sum of at most $15 - 6 = 9$ polynomials. This suggests, following the ideas in \cite{scheiderer}, that the Pythagoras number of $H_{5,4}$ is upper bounded by 10, which is smaller than the known upper bound 11 mentioned in that paper.
\end{remark}

\subsection{Sharp bounds for \texorpdfstring{$d = 3$}{d = 3}}

\begin{corollary}
\label{coro:d3}
  The maximum number of linearly independent polynomials in a SOS decomposition of a strictly positive polynomial in the boundary of $\Sigma_{n,6}$, $n \ge 3$, is $\dim(H_{n,2}) - 7$ and this bound is sharp for all $n \ge 3$.
\end{corollary}
\begin{proofels}
For $n = 3$, the dimension of $H_{3,3}$ is 10. G. Blekherman \cite{blekherman} provides a direct way to construct examples of strictly positive polynomials in $\partial \Sigma_{3,6}$ which is the sum of squares of 3 linearly independent polynomials, hence the bound is sharp in this case. By Proposition \ref{prop:sharp}, it is then sharp for all $n \ge 3$.
\end{proofels}

We next show how to construct examples satisfying the bound for $n=3$  and any $d \ge 4$, proving the sharpness of the bound for any $(n,d)$, $n \ge 3$, $d \ge 4$.

\subsection{Sharp bounds for \texorpdfstring{$d \ge 4$}{d >= 4}}
\label{subsection:d4plus}

We use \cite[Theorem 21]{BSV} restated in our setting:

\begin{theorem}
\label{teo:gamma}
Let $W \subset H_{n,d}$ be a set of $\dim(H_{n,d}) - p - 1$ linearly independent homogeneous polynomials such that $\langle W \rangle$ is a zero-dimensional ideal. If $\Gamma = V(W)$ consist of $p+2$ linearly dependent real points (that is, there exist coefficients $\{u_\gamma\}_{\gamma \in \Gamma}$, not all zero, such that $\sum_{\gamma \in \Gamma} u_\gamma q(\gamma) = 0$ for all homogeneous polynomials $q \in H_{n,d}$), then the Hankel index $\Hankel(n,d) \le p$.
\end{theorem}

From the proof of \cite[Theorem 21]{BSV} given by the authors, one can obtain a refinement of the latter result, which provides an easier way to construct examples.

\begin{theorem}
\label{teo:gammaprime}
Let $W \subset H_{n,d}$ be as before and let $\Gamma = V(W)$. If there exists a minimal linearly dependent subset $\Gamma' \subset \Gamma$ of cardinal $k+2$, then $\Hankel(n,d) \le k$.

\end{theorem}

In our setting, if we know the coefficients of the linear dependence relation, we can deduce from \cite[Theorem 6.1]{blekherman} and \cite[Theorem 21]{BSV} a formula to construct explicit examples.

\begin{theorem}
\label{thm:formula}
Let $\Gamma' = \{v_1, \dots, v_{k+2}\}$ be as before, and let
$$
\sum_{v \in \Gamma'} u_v q(v) = 0
$$
be the unique (up to scaling) relation. Define
$$
Q(h) = a_1 h(v_1)^2 + \dots + a_{k+2} h(v_{k+2})^2, \quad \text{ for } h \in H_{n,d},
$$
with a single negative coefficient $a_j$, the rest $a_i$ positive, and $\sum_{i=1}^{k+2} \frac{u_i^2}{a_i} = 0$.
Then $Q$ has kernel of rank $M = \dim(H_{n,d}) - k$. If $\ker(Q) = \langle q_1, \dots, q_{M} \rangle$, then
$f = q_1^2 + \dots + q_{M}^2$ is a strictly positive polynomial in the boundary of $\Sigma_{n,2d}$.
\end{theorem}

\subsubsection{Sharp bounds for \texorpdfstring{$d = 4$}{d = 4}}
\label{subsection:d4}

We will start by constructing an example for the specific case of $d = 4$. Afterwards, we will demonstrate that this example can be generalized for any $d \ge 4$.

By Section \ref{section:Hankel}, the maximum number of polynomials in a SOS decomposition of a polynomial of degree $2d = 8$ is $\dim H_{n,4} -10$. By Corollary \ref{coro:n0n}, if the bound is attained for $n = 3$, then it is attained for all $n \ge 3$.

For $n = 3$, the dimension of $H_{3,4}$ is $\binom{6}{2} = 15$ so the maximum number of polynomials that can appear  in a sum of squares decomposition of a polynomial on the positive boundary of $\Sigma_{3,8}$ is $15 - 10 = 5$.

We want to prove $\Hankel(3,4) \le 10$. To use Theorem \ref{teo:gamma} directly,  we should take $p = 10$ in that theorem, and since $\dim(H_{3,4}) = 15$, we need to find 4 linearly independent polynomials of degree 4 intersecting in 12 linearly dependent points. Instead of doing that, following Theorem \ref{teo:gammaprime}, we will look for two polynomials of degree 4 intersecting in a set $\Gamma$ of 16 real points such that there exists a subset of 12 roots $\Gamma' = \{v_1, \dots, v_{12}\}$ satisfying a unique relation of linear dependence for a generic $h \in H_{3,4}$:
\begin{equation} \label{linearRelation}
u_1 h(v_1) + \dots + u_{12} h(v_{12}) = 0,
\end{equation}
for $(u_1, \dots, u_{12}) \in \R^{12}$.

We consider $\Gamma$ the set of common roots of
$$
\begin{aligned}
q_1 &= (x_1+x_0)(x_1)(x_1-x_0)(x_1-2x_0), \\
q_2 &= (x_2)(x_2-x_0)(x_2 - 2x_0)(x_2-3x_0),
\end{aligned}
$$
and let $\Gamma'$ be the subset containing $12$ points in 3 columns of $4$ points:
$$
\Gamma' = \{ (1 : a : b) : a \in \{-1,0,1\}, b \in \{0,1,2,3\}\} \subset \Proj^2
$$

By a direct computation, we can verify that these points satisfy the following relation
$$
\sum_{i = 0}^3 (-1)^i \binom{3}{i} h(1,-1,i) - 2 \sum_{i = 0}^3 (-1)^i \binom{3}{i} h(1,0,i) + \sum_{i = 0}^3 (-1)^i \binom{3}{i} h(1,1,i) = 0,
$$
for any $h \in H_{3,4}$. That is, $u = (1, -3, 3, 1, -2, 6, -6, 2, 1, -3, 3, 1)$ in \eqref{linearRelation} (see Lemma \ref{lemma:bivariate} for the general construction).


To construct an explicit example, following Theorem \ref{thm:formula}, we set $a_i = u_i^2$ for $1 \le i \le 11$ and $a_{12} = -\frac{1}{11}$, and define $Q: H_{3,4} \rightarrow \R$, $Q(h) = \sum_{i=1}^{12} a_i h(v_i)^2$. This form must have kernel of dimension 5. By a Maple computation (see \cite[Worksheet A]{strictlyPositive}) we obtain that the kernel is generated by
$$
\begin{aligned}
w_1 &= q_2 = (x_2)(x_2-x_0)(x_2 - 2x_0)(x_2-3x_0), \\
w_2 &= x_0(x_1^3-x_1x_0^2),\\
w_3 &= x_1(x_1^3-x_1x_0^2), \\
w_4 &= x_2(x_1^3-x_1x_0^2), \\
w_5 &= -x_0^4+x_0^3x_2+3x_0^2x_1^2+12x_0^2x_1x_2+3x_0^2x_2^2-15x_0x_1^2x_2-18x_0x_1x_2^2- \\
& - 2x_0x_2^3+9x_1^2x_2^2+6x_1x_2^3,
\end{aligned}
$$
hence $f = w_1^2 + \dots + w_5^2$ is a strictly positive polynomial in the boundary of $\Sigma_{3,8}$ (which can also be verified numerically using SEDUMI).

\begin{remark}
We observe that $\{w_1, w_2, w_3, w_4\}$ is a set of 4 linearly independent polynomials intersecting in 12 points, which satisfy the hypothesis of \cite[Theorem 21]{BSV}.
\end{remark}

Since the Ottaviani-Paoletti conjecture is known to be true for the case $d=4$, the example gives a complete proof of the following result.

\begin{proposition}
For $d = 2$ and $n \ge 4$ the maximum number of polynomials that can appear in a SOS decomposition of a strictly positive polynomial in the boundary of $\Sigma_{n,4}$ is $\dim H_{n,4} - 6$. For $d = 3$ or $d = 4$ and $n \ge 3$, the maximum number of polynomials that can appear in a SOS decomposition of a strictly positive polynomial in the boundary of $\Sigma_{n,2d}$ is $\dim H_{n,2d} - 6$. All these bounds are sharp.
\end{proposition}

\subsubsection{Sharp bounds for \texorpdfstring{$d \ge 5$}{d >= 5}}

We can generalize the example in the previous section to any $d \ge 4$ in the following way.

We take $\Gamma$ the set of common roots of
\begin{align*}
q_1 &= (x_1+x_0)(x_1)(x_1-x_0)\cdots (x_1-(d-2)x_0), \\
q_2 &= (x_2)(x_2-x_0)(x_2 - 2x_0) \cdots (x_2-(d-1)x_0)
\end{align*}
and let $\Gamma'$ be a subset  containing $3d$ points in 3 columns of $d$ points:
$$
\Gamma' = \{ (1 : a : b) : a \in \{-1,0,1\}, b \in \{0,1,\dots,d-1\}\} \subset \Proj^2.
$$

We use the following lemma to deduce a general relation among the points.

\begin{lemma}
Let $d \in \N$, then
$$
\phi_{d,a}(x) = \sum_{k = 0}^d (-1)^k \binom{d}{k} (x+k)^a = 0
$$
for all $a = 0, 1, \dots, d-1$ and all $x \in \R$.
\end{lemma}

\begin{proofels}
The formula corresponds to the $d$th order forward difference $\Delta^n[f](x)$ applied to $f(x) = x^a$, and this function vanishes on all polynomials of degree smaller than or equal to $d - 1$ (see for example \cite[Figure 23.1]{ChapraCanale}).
\end{proofels}

We can now establish the relation between the evaluations at the points of $\Gamma'$.

\begin{lemma}
\label{lemma:bivariate}
Let $m \in H_{3,d}$ be a monomial, $m = x_0^a x_1^b x_2^c$, with $a,b,c \in \Z_{\ge 0}$ and $a+b+c = d$. Define
$$\Phi_{\alpha}(m) := \sum_{k = 0}^{d-1} (-1)^k \binom{d-1}{i} m(1,\alpha,k).$$
Then
\begin{equation}\label{uniqueRelation}
\Phi_{-1}(m) - 2 \Phi_0(m) + \Phi_1(m) = 0.
\end{equation}
\end{lemma}
\begin{proofels}
If $c < d-1$, the previous lemma implies that $\Phi_{\alpha}(m) = 0$ for all $\alpha \in \R$, and the identity follows. For $c \ge d-1$, it remains to verify the identity for the monomials
$$
x_2^d, \quad x_0 x_2^{d-1} \quad \text { and } \quad x_1 x_2^{d-1},
$$
which can be verified by a direct computation.
\end{proofels}

\begin{remark}
Formula \ref{uniqueRelation} is actually a bivariate forward difference of order $2$ in $x_1$ and order $d-1$ in $x_2$, hence it vanishes on all monomials $m(x_0,x_1,x_2)$ of degree $< 2$ in $x_1$ or degree $< d-1$ in $x_2$, as required.
\end{remark}

\begin{corollary}
\label{coro:relation}
Let $\Gamma = \{ (1 : a : b) : a \in \{-1,0,1\}, b \in \{0,1,\dots, d-1\}\} \subset \Proj^2$  and let
$$
u_{(1:a:b)} = \begin{cases}
(-1)^b \binom{d-1}{b} & \text{ if }a = -1 \text{ or } a = 1, \\
(-2) (-1)^b \binom{d-1}{b} & \text{ if }a = 0.
\end{cases}.$$
Then $\sum_{v \in \Gamma} u_v h(v) = 0$ for all homogeneous polynomials $h \in \R[x_0, x_1, x_2]$ of degree $d$.
\end{corollary}

We again follow formula \cite[Theorem 6.1]{blekherman} to construct examples.
\begin{lemma}
\label{lemma:d=3}
Let $\Gamma \subset \Proj^2$ and $u_v$, $v \in \Gamma$, be as in Corollary \ref{coro:relation}. Define the quadratic form $Q: H_{3,d} \rightarrow \R$,
$$Q(h) = a_1 h(v_1)^2 + \dots + a_{3d} h(v_{3d})^2$$
with $a_i = u_i^2$ for $1 \le i \le 3d-1$ and $a_{3d} = -\frac{1}{3d-1}$. Then $Q$ has kernel of dimension $M = \dim(H_{n,d}) - (3d - 2)$. If the kernel is generated by $\{w_1, \dots, w_M\}$, then $f = w_1^2 + \dots + w_M^2$ is a strictly positive polynomial in the boundary of $\Sigma_{3,2d}$.
\end{lemma}

\begin{example}
Applying this strategy to the case $d = 5$ and $\Gamma = \{ (1 : a : b) : a \in \{-1,0,1\}, b \in \{-2, -1, 0, 1, 2\}\} \subset \Proj^2$, we obtain a form $Q$ with kernel of dimension 8, generated by
$$
\begin{aligned}
w_1 &= 4 x_0^4 x_2-5 x_0^2 x_2^3+x_2^5, \\
w_2 &= \frac{4}{15} x_0^5-\frac{4}{5} x_0^3 x_1^2+2 x_0^3 x_1 x_2-x_0^3 x_2^2+2 x_0^2 x_1^2 x_2-x_0^2 x_1 x_2^2+2 x_0 x_1^2 x_2^2- \\
& - 2 x_0 x_1 x_2^3+\frac{1}{3} x_0 x_2^4-2 x_1^2 x_2^3+x_1 x_2^4, \\
w_3 &= -x_0^2 x_1 x_2^2+x_1^3 x_2^2, \\
w_4 &= -x_0^2 x_1^2 x_2+x_1^4 x_2, \\
w_5 &= -x_0^4 x_1+x_1^5, \\
w_6 &= -x_0^3 x_1 x_2+x_0 x_1^3 x_2, \\
w_7 &= -x_0^3 x_1^2+x_0 x_1^4, \\
w_8 &= -x_0^4 x_1+x_0^2 x_1^3,
\end{aligned}
$$
hence $f = w_1^2 + \dots + w_{8}^2$ is a strictly positive polynomial in the boundary of $\Sigma_{3,10}$ (see \cite[Worksheet B]{strictlyPositive}).
\end{example}

We obtain the following result which proves that the bounds are sharp in all cases.

\begin{proposition}
For $d = 2$ and $n \ge 4$, there exists a strictly positive polynomial in the boundary of $\Sigma_{n,4}$ that is the sum of squares of $\dim H_{n,2} - 6$ polynomials. For $d \ge 3$ and $n \ge 3$, there exists a strictly positive polynomial in the boundary of $\Sigma_{n,2d}$ that is the sum of squares of $\dim H_{n,2d} - (3d-2)$ polynomials.
\end{proposition}

\begin{proofels} We have already proved the cases $d = 2$ and $d = 3$. For $n = 3$ and $d \ge 4$ we use the construction in Lemma \ref{lemma:d=3}, and we extend this example to any $n \ge 3$ using Corollary \ref{coro:n0n}.\end{proofels}

We remark that all the polynomials in the proposition can be effectively computed using a computer algebra system.

\subsubsection{General configuration of points}

In the constructions above we used a specific configuration of points which allowed us to explicitly compute the vector $u$ of coefficients of the unique relation $u_1 h(v_1) + \dots + u_{3d} h(v_{3d}) = 0$. The existence of such relation can be deduced in a more general setting from one of the Cayley-Bacharach theorems which we recall here.

\begin{theorem}[{\cite[Theorem CB4]{eisenbud96}}] \label{thm:CB4} Let $X_1, X_2 \subset \Proj^2$ be plane curves of degrees $d$ and $e$ respectively, meeting in a collection of
$d \cdot e$ distinct points $\Gamma = \{p_1,... ,p_{d \cdot e}\}$.  Any set $\Gamma'$ consisting of all but one point of $\Gamma$ impose independent conditions on
forms of degree $d + e - 3$. If $C \subset \Proj^2$ is any
plane curve of degree $d + e - 3$ containing all but one point of $\Gamma$, then $C$ contains all of $\Gamma$.
\end{theorem}

Our claim is now a direct corollary.

\begin{proposition}
Let $f \in \R[x_0,x_1,x_2]$ be a homogenous polynomial of degree 3 and $g \in \R[x_0,x_1,x_2]$ a homogenous polynomial of degree $d \ge 1$ such that $\Gamma = \V(f) \cap \V(g)$ contains $3d$ real distinct points. There exist a unique (up to scaling) relation
$$
u_1 h(v_1) + \dots + u_{3d} h(v_{3d}) = 0
$$
for all polynomials $h\in \R[x_0,x_1,x_2]$ of degree $d$.
\end{proposition}

\begin{proofels} Take $e = 3$ in Theorem \ref{thm:CB4} and let $h$ be a form of degree $d$. By the second claim of the theorem, the linear forms $h(v_1), \dots, h(v_{3d})$ on the coefficients of $h$ are linearly dependent. If the relation is not unique, we can construct a relation involving $3d-1$ points, which contradicts the first claim of the theorem.
\end{proofels}

Therefore, given a polynomial $f$ of degree $3$ and a polynomial of degree $d$ intersecting in $3d$ real points, we can apply Theorem \ref{thm:formula} to construct a strictly positive polynomial in the boundary of $\Sigma_{n,2d}$ that is the sum of squares of $\dim{H_{n,d}} - (3d-2)$ linearly independent polynomials of degree $d$.

\section{Minimal number of polynomials in any SOS decomposition}

We have studied in the previous section the maximum number of polynomials in a SOS decomposition of a strictly positive polynomial in the boundary of the SOS cone.
We now study a related question. For fixed number of variables $n$ and degree $d$, which is the minimal number $s \in \N$ such that any strictly positive polynomial in the boundary of $\Sigma_{n,2d}$ can be decomposed as the sum of squares of $s$ polynomials?

The sum of squares length of a polynomial $f \in \Sigma_{n,2d}$ is defined as
$$
\ell(f) = \min\{r \ge 0 : \exists p_1, \dots, p_r \in H_{n,d} \text{ with } f = p_1^2 + \dots + p_r^2\}.
$$
Following this definition, our question can be rephrased as finding the maximum SOS length of the polynomials in $\partial\Sigma_{n,2d}^+$.

This question is related to the algebraic description of the boundary. For example, in \cite{blekherman2} the authors use these bounds to get properties of the variety defined as the algebraic closure of the boundary for the cases $(n, 2d) = (4,4)$ and $(3,6)$. In these two cases, every strictly positive polynomial in the boundary of $\Sigma_{n,2d}$ has length $n$.

In the general case, by dimension arguments that we will explain below, one can easily get bounds for the minimum number of polynomials in a decomposition of a polynomial in $\partial\Sigma_{n,2d}$. In particular, one can show that $n$ polynomials are in general not enough, that is, there must be polynomials in $\partial\Sigma_{n,2d}$ of length greater than $n$. One natural question is then if the bounds given by the dimension count are optimal. We will provide examples of polynomials in $\partial\Sigma_{n,2d}^+$ of length larger than the predicted number, giving a negative answer to this question.

For a given polynomial it is not easy to compute its SOS length. Numerical solvers based on interior point methods will usually compute decompositions with large number of polynomials. On the other hand, exact solvers are very expensive computationally and can only solve small problems. A usually good heuristics to find decompositions with small number of polynomials is to minimize the nuclear norm of the SDP problem using a numerical solver (see for example \cite[Section 2.2.6]{BPT}). However there is no easy way to prove that the decompositions found by these methods are indeed minimal.
A case where it is easy to prove minimality is the case where the decomposition is unique, so we look for these type  of examples.

\subsection{Polynomials in \texorpdfstring{$\partial\Sigma_{4,6}^+$}{Sigma(4,6)} with unique SOS decompositions}
\label{subsection:46}

We study in this section the case of strictly positive polynomials in the boundary of $\Sigma_{4,6}$.

\begin{example}[sum of 4 squares with unique decomposition]\label{ex:reznick} The first examples known to us of strictly positive polynomials in $\partial \Sigma_{4,6}$ appeared in \cite{reznick}. The authors provided a construction for strictly positive polynomials of degree 6 and any number of variables that are in the boundary of the corresponding SOS cone. For 4 variables, one of such examples is $f = p_1^2 + p_2^2 + p_3^2 + p_4^2$, with
\begin{align*}
p_1 &= x((2-1/2)x^2-(y^2+z^2+w^2)), \\
p_2 &= y((2-1/2)y^2-(x^2+z^2+w^2)), \\
p_3 &= z((2-1/2)z^2-(x^2+y^2+w^2)), \\
p_4 &= w((2-1/2)w^2-(x^2+y^2+z^2)).
\end{align*}

Using SEDUMI we can verify numerically that in this case the provided decomposition is the unique decomposition of $f$ as a sum of squares of polynomials (modulo orthogonal transformations). Equivalently, the Gram spectrahedron of $f$ consists of one point. This is a similar situation as in the case $(n,2d)=(3,6)$, for which it is proven in \cite{capco} that all strictly positive polynomials in the boundary of the SOS cone have a unique SOS decomposition. See \cite[Worksheet C]{strictlyPositive}.
\end{example}

We now provide a rigorous prove that the decomposition in Example \ref{ex:reznick} is indeed unique.
The strategy followed up here was suggested by C. Scheiderer.

\begin{proposition}
\label{prop:unique46}
 Let $p_1, p_2, p_3, p_4$ and $f$ be as in Example \ref{ex:reznick}. The Gram spectrahedron of $f$ consists of a unique point, and hence $f$ allows a unique sum of squares decomposition modulo orthogonal transformations.
\end{proposition}

\begin{proofels} See \cite[Worksheet C]{strictlyPositive} for the computations.
If a linear functional $\ell \in \Sigma_{4,6}^*$ vanishes in $f$ then, since $f = p_1^2 + \dots + p_4^2$, $\ell(p_i^2) = 0$ for $1 \le i \le 4$  and hence $p_1, \dots, p_4$ are in the kernel $W_{\ell}$ of the associated quadratic form $Q_{\ell}: H_{4,3} \rightarrow \R$, $Q_{\ell}(p) = \ell(p^2)$. To prove that any polynomial in a SOS decomposition of $f$ is in the span of $\{p_1, \dots, p_4\}$, we look for $\ell \in \Sigma_{4,6}^*$ such that $W_{\ell}$ is generated by $\{p_1, p_2, p_3, p_4\}$.

By \cite[Lemma 2.6]{blekherman}, $\ell$ must satisfy $\ell(p_i q) = 0$ for all $q \in H_{4,3}$ and all $p_i$, $1 \le i \le 4$.
Note that $\ell \in \Sigma^*_{4,6} \subset H^*_{4,6}$, which is an 84-dimensional space. The dimension of $H_{4,3}$ is 20, so for each $p_i$ we get 20 restrictions for $\ell$. This would be a total number of $80$ restrictions, but noting that $p_i p_j = p_j p_i$, we are imposing at most $20 + 19 + 18 + 17 = 74$ restrictions. Doing the computations in Maple we obtain that the space of linear functionals satisfying all the above conditions is indeed a $84-74 = 10$-dimensional space.

We need to find a functional $\ell$ in this $10$-dimensional space such that the associated quadratic form $Q_{\ell}$ is positive semidefinite. This is not easy to obtain by trial and error, so we try instead to find $Q_{\ell}$ as a sum of positive semidefinite forms. To do this, we consider first $p_5 = x^3$, and repeat the above procedure to compute a linear functional $\ell_1$ such that the kernel of $Q_{\ell_1}$ contains $\{p_1, \dots, p_5 \}$. We now have at most $20+19+18+17+16 = 90$ restrictions, but it turns out that there are other dependencies in the restrictions, and there is a unique $\ell_1$ (up to a scalar multiple) satisfying all the conditions. Moreover choosing the correct sign of the scalar multiple, the form $Q_{\ell_1}$ is positive semidefinite. In an analogous way, replacing $x^3$ by $y^3$, $z^3$ and $w^3$, we obtain linear functionals $\ell_2$, $\ell_3$, $\ell_4$ such that the associated quadratic forms are positive semidefinite with kernels generated by $\{p_1, \dots, p_4 \}$ and the respective new polynomial.

The last step is to consider $\ell = \ell_1 + \ell_2 + \ell_3 + \ell_4$. We know that for such $\ell$, the quadratic form $Q_{\ell}$ will be positive semidefinite (because it is the sum of positive semidefinite forms) with kernel equal to the intersection of the kernels of each $Q_{\ell_i}$. The kernel of this form $\ell$ turns out to be $W_{\ell}$, hence it satisfies the desired properties.

To conclude the proof, note that if $f = q_1^2 + \dots + q_s^2$ is another SOS decomposition of $f$, then $q_i \in W_{\ell}$, for all $1 \le i \le s$. Since $\dim(W_{\ell}) = 4$, which coincides with  the number of variables in the ring, following the proof of \cite[Corollaries 1.3 and 1.4]{blekherman} we obtain that $f$ cannot be written as the sum of less than 4 squares. And then, by \cite[Lemma 4.3]{capco}, $f$ has a unique representation as a sum of 4 squares (up to orthogonal equivalence).
\end{proofels}

We remark that the form $\ell$ in the proof of Proposition \ref{prop:unique46} is not an extreme ray in the normal cone of $f$, $N_f = \{\alpha \in \Sigma_{4,6}^*: \alpha(f) = 0\}$, since it is the sum of the forms $\ell_1$, $\ell_2$, $\ell_3$ and $\ell_4$; but the latter ones are extreme rays in that cone. Indeed, if any of these forms is a convex combination of other forms, then they all must have the same kernel and hence they are all a scalar multiple of each other, as in \cite[Lemma 2.2]{blekherman}.

This is different from the cases $(n,2d)=(4,4)$ and $(3,6)$, for which C. Scheiderer and J. Capco show in \cite[Corollary 3.7]{capco} that for every strictly positive polynomial $f$ in the boundary of the SOS cone, the normal cone $N_f$ consists of a single ray.


%

The above example is a sum of squares of a sequence of parameters. We now give an example that is a sum of squares of 5 polynomials, and hence the polynomials are not a sequence of parameters.

\begin{example}[sum of 5 squares with unique decomposition]
\label{ex:46sumOf5}
 Let $p_1, p_2, p_3, p_4$ be as in Example \ref{ex:reznick}, and let $p_5 = yzw$. For $f = p_1^2 + p_2^2 + p_3^2 + p_4^2 + p_5^2$ we can verify using SEDUMI or constructing an appropriate quadratic form that the given decomposition is the unique decomposition of $f$ as sum of squares of polynomials (see \cite[Worksheet C]{strictlyPositive} for the construction of the quadratic form as in Proposition \ref{prop:unique46}).
\end{example}

\subsubsection{\texorpdfstring{The algebraic boundary of $\Sigma_{4,6}$.}{The algebraic boundary of Sigma(4,6)}}
\label{section:algBound46}

One motivation for studying these examples is finding equations for the union $S$ of the non-discriminant components of the algebraic boundary of $\Sigma_{4,6}$. We recall that $\Sigma_{4,6}$ is a full-dimensional cone in $H_{4,6}$ and hence the dimension of $\Sigma_{4,6}$ is $\binom{9}{3} = 84$ and the dimension of $\partial\Sigma_{4,6}$ is 83.

Let $V = \{p_1^2 + \dots + p_5^2: p_i \in H_{4,3}\}$ be the variety of polynomials in $H_{4,6}$ that can be decomposed as a sum of 5 squares. The dimension of $H_{4,3}$ is 20 and the dimension of the orthogonal group $O(5)$ is $\binom{5}{2} = 10$. So, by orthogonal equivalence, the total dimension of $V$ is at most $20 \cdot 5 - 10 = 90$. Hence, by comparing only the dimensions, it is in principle possible that every polynomial in $\partial\Sigma_{4,6}$ is the sum of 5 squares. However, we show now an example of a polynomial in that variety that is the sum of 6 squares and cannot be decomposed as the sum of 5 squares.

\begin{example}[sum of 6 squares with unique decomposition]
\label{ex:46sumOf6}
Let $p_1, p_2, p_3, p_4$ be as in Example \ref{ex:reznick}, and let
$$p_5 = xyz + xzw \quad \text{and} \quad p_6 = x^2y + xy^2.$$
For $g = p_1^2 + p_2^2 + p_3^2 + p_4^2 + p_5^2 + p_6^2$ we verify using SEDUMI that the decomposition is unique (see \cite[Worksheet C]{strictlyPositive}). For this, we maximize different linear function over the Gram spectrahedron of $g$ and verify that it consists of only one point. We will show in Example \ref{ex:54sumOf7} a strategy that could also be applied here to give a rigourous proof that the decomposition is unique. However, the computations in this case are more complicated and we do not pursue this strategy.
\end{example}

\subsection{Polynomials in \texorpdfstring{$\Sigma_{5,4}$}{Sigma(5,4)} with unique SOS decompositions}
\label{subsection:54}

We provide examples of strictly positive polynomials in the boundary of $\Sigma_{5,4}$ for which the decomposition is unique.

\begin{example}[sum of 5 squares with unique decomposition]
\label{ex1:54}
We start with an example in the boundary of $\Sigma_{4,4}$ (which can be obtained following the construction in \cite{blekherman}) and add a fifth polynomial to it. Let $f = p_1^2+p_2^2+p_3^2+p_4^2+p_5^2$, with
\begin{align*}
p_1 &= x_1^2- x_4^2, \\
p_2 &= x_2^2- x_4^2, \\
p_3 &= x_3^2- x_4^2, \\
p_4 &= -x_1^2 - x_1 x_2 - x_1 x_3 + x_1 x_4 - x_2 x_3 + x_2 x_4 + x_3 x_4, \\
p_5 &= x_5^2.
\end{align*}

It can be easily seen numerically using SEDUMI that the polynomial $f$ is a strictly positive polynomial in the boundary of $\Sigma_{5,4}$ and that the Gram spectrahedron of $f$ consists of a unique matrix of rank 5. Thus the decomposition is unique modulo orthogonal transformations.
See \cite[Worksheet D]{strictlyPositive}.
\end{example}

\begin{example}[sum of 6 squares with unique decomposition]
\label{ex3:54}
We add a new polynomial $p_6$ to the latter example. Let $f = p_1^2+p_2^2+p_3^2+p_4^2 + p_5^2 + p_6^2$, with
\begin{align*}
p_1 &= x_1^2- x_4^2, \\
p_2 &= x_2^2- x_4^2, \\
p_3 &= x_3^2- x_4^2, \\
p_4 &= -x_1^2 - x_1 x_2 - x_1 x_3 + x_1 x_4 - x_2 x_3 + x_2 x_4 + x_3 x_4, \\
p_5 &= x_5^2, \\
p_6 &= x_1 x_5 + x_4 x_5.
\end{align*}

Using SEDUMI, it can be easily verified numerically that the polynomial $f$ is a strictly positive polynomial in the boundary of $\Sigma_{5,4}$ and that the Gram spectrahedron of $f$ consists of a unique matrix of rank $6$  (see \cite[Worksheet D]{strictlyPositive}).
\end{example}

\subsubsection{\texorpdfstring{The algebraic boundary of $\Sigma_{5,4}$.}{The algebraic boundary of Sigma(5,4)}}

As in the case $\Sigma_{4,6}$, we are interested in finding equations for the union $S$ of the non-discriminant components of the algebraic boundary of $\Sigma_{5,4}$. The cone $\Sigma_{5,4}$ is a full-dimensional cone in $H_{5,4}$ and hence the dimension of $\Sigma_{5,4}$ is $\binom{8}{4} = 70$ and the dimension of $\partial\Sigma_{5,4}$ is 69.

Let $V = \{p_1^2 + \dots + p_6^2: p_i \in H_{5,2}\}$ be the variety of polynomials in $H_{5,4}$ that can be decomposed as a sum of 6 squares. The dimension of $H_{5,2}$ is 15. So, as in Section \ref{section:algBound46}, the total dimension of $V$ is at most $15 \cdot 6 - \binom{6}{2} = 75$. Comparing the dimensions obtained, we see that it is in principle possible that every polynomial in $\partial\Sigma_{5,4}$ is a sum of 6 squares.
However, we show now an example of a polynomial in that variety that is the sum of 7 squares and cannot be decomposed as the sum of 6 squares.

\begin{example}
\label{ex:54sumOf7}
Let $p_1, p_2, p_3, p_4$ be as in Example \ref{ex1:54} and set
$$p_5 = x_1x_5 + x_2x_5, \quad p_6 = x_3 x_5 + x_2 x_5 \quad \text{ and } \quad p_7 = x_5^2.$$

Let $g = p_1^2+p_2^2+p_3^2+p_4^2+p_5^2+p_6^2+p_7^2$. Then $g$ is a strictly positive polynomial in the boundary of $\Sigma_{5,4}$ and it cannot be decomposed as the sum of 6 squares.
\end{example}

In this case we are able to completely prove this statement, following here the strategy used in \cite{laplagne8}.

\begin{proposition}
Let $g = p_1^2+p_2^2+p_3^2+p_4^2+p_5^2+p_6^2+p_7^2$ from Example \ref{ex:54sumOf7}. Then $g$ is a strictly positive polynomial in the boundary of $\Sigma_{5,4}$ and it cannot be decomposed as the sum of 6 squares.
\end{proposition}

\begin{proofels}
See \cite[Worksheet D]{strictlyPositive}).
The polynomial $g$ is strictly positive, since the only common root of $\{p_1, p_2, p_3, p_4\}$ in $\R^4$ is $(x_1, x_2, x_3, x_4) = (0,0,0,0)$ and this forces $x_5 = 0$ in $p_7$.

To prove that it cannot be decomposed as the sum of 6 squares, we construct first a linear form $\ell \in \Sigma_{5,4}^*$ that vanishes in $g$. As in
 the proof of Proposition \ref{prop:unique46}, the form $\ell$ must satisfy $\ell(p_i q) = 0$ for any $q \in H_{5,2}$. In this case, the form $\ell$ is unique (up to linear scaling) and the quadratic form $Q_{\ell}$ has kernel
 $$W = \langle p_1, p_2, p_3, p_4, x_1x_5, x_2x_5, x_3x_5, x_4x_5, x_5^2\rangle.$$
 That is, every polynomial in a SOS decomposition of $g$ is a linear combination of these 9 polynomials. This proves that $g$ is in the boundary of $\Sigma_{5,4}$, because its Gram spectrahedron is not of full rank.

To prove that $g$ cannot be decomposed as the sum of 6 squares, we define 6 polynomials $q_1, \dots, q_6 \in W$ with generic coefficients. As in \cite{laplagne8} we can assume that the 6 polynomials are triangulated. That is, we assume
{\small
$$\arraycolsep=.4pt
\begin{array}{r*{18}{l}}
q_1 &= a_{11} p_1 &+ &a_{21} p_2 &+ &a_{31} p_3 &+ &a_{41} p_4 &+ &a_{51} x_1x_5 &+ &a_{61} x_2x_5 &+ &a_{71} x_3x_5 &+ &a_{81} x_4x_5 &+ &a_{91} x_5^2 \\
q_2 &=            &  &a_{22} p_2 &+ &a_{32} p_3 &+ &a_{42} p_4 &+ &a_{52} x_1x_5 &+ &a_{62} x_2x_5 &+ &a_{72} x_3x_5 &+ &a_{82} x_4x_5 &+ &a_{92} x_5^2, \\
q_3 &=            &  &           &  &a_{33} p_3 &+ &a_{43} p_4 &+ &a_{53} x_1x_5 &+ &a_{63} x_2x_5 &+ &a_{73} x_3x_5 &+ &a_{83} x_4x_5 &+ &a_{93} x_5^2, \\
q_4 &=            &  &           &  &           &  &a_{44} p_4 &+ &a_{54} x_1x_5 &+ &a_{64} x_2x_5 &+ &a_{74} x_3x_5 &+ &a_{84} x_4x_5 &+ &a_{94} x_5^2, \\
q_5 &=            &  &           &  &           &  &           &  &a_{55} x_1x_5 &+ &a_{65} x_2x_5 &+ &a_{75} x_3x_5 &+ &a_{85} x_4x_5 &+ &a_{95} x_5^2, \\
q_6 &=            &  &           &  &           &  &           &  &              &  &a_{66} x_2x_5 &+ &a_{76} x_3x_5 &+ &a_{86} x_4x_5 &+ &a_{96} x_5^2.
\end{array}
$$
}

The equation $q_1^2 + \dots + q_6^2 = g$ defines a system of quadratic equations in the coefficients $\{a_{ij}\}_{1 \le i \le 9, 1 \le j \le 6, j \le i}$. To solve this system, we compute in Singular a Groebner basis of the ideal defined by the equations using the degree reverse lexicographical ordering. From the equations obtained we derive the following restrictions:
$$
\begin{aligned}
0 &= a_{96} = a_{95} \\
&= a_{84}= a_{74} = a_{64} = a_{54} \\
&= a_{83} = a_{73} = a_{63} = a_{53} = a_{43}\\
&= a_{82} = a_{72} = a_{62} = a_{52} = a_{42} = a_{32} \\
& = a_{81} = a_{71} = a_{61} = a_{51} = a_{41} = a_{31} = a_{21} \\
\end{aligned}
$$
$$
a_{75} = a_{85}
$$
$$
a_{11}^2 = a_{22}^2 = a_{33}^2 = a_{44}^2 = 1
$$

Since we can assume that the initial coefficient in each polynomial is positive, we simplify the last four restrictions to $a_{11} = a_{22} = a_{33} = a_{44} = 1$.

After applying these subsitutions in the original polynomials, we obtain the following polynomials:
$$\arraycolsep=.4pt
\begin{array}{r*{4}{l}}
q_1 &=        p_1 &+ &a_{91} x_5^2 \\
q_2 &=        p_2 &+ &a_{92} x_5^2, \\
q_3 &=        p_3 &+ &a_{93} x_5^2, \\
q_4 &=        p_4 &+ &a_{94} x_5^2,
\end{array}
$$
$$\arraycolsep=.4pt
\begin{array}{r*{8}{l}}
q_5 &=  a_{55} x_1x_5 &+ &a_{65} x_2x_5 &+ &a_{75} x_3x_5 &+ &a_{75} x_4x_5, \\
q_6 &=                &  &a_{66} x_2x_5 &+ &a_{76} x_3x_5 &+ &a_{86} x_4x_5.
\end{array}
$$

The equation $q_1^2 + \dots + q_6^2 = g$ defines a new system of quadratic equations in the coefficients. To solve this system, we apply the factorizing Groebner basis algorithm (with lexicographical ordering) to the ideal $I$ generated by the equations. This gives a set of 3 ideals $I_1, I_2, I_3$ whose intersection has the same zero--set as the input ideal (that is, $\sqrt{I} = \sqrt{I_1 \cap I_2 \cap I_3}$). We find that the following polynomials belong to these ideals:
{\small
$$
\begin{aligned}
I_1 &\ni 49a_{94}^6-28a_{94}^5+4a_{94}^3+10a_{94}^2+2, \\
I_2 &\ni 105a_{94}^8+232a_{94}^7+28a_{94}^6-312a_{94}^5-232a_{94}^4+104a_{94}^3+204a_{94}^2+96a_{94}+16, \\
I_3 &\ni 16a_{93}^6+8a_{93}^5-12a_{93}^4+4a_{93}^3+4a_{93}^2-4a_{93}+1.
\end{aligned}
$$
}

We check in Maple that these polynomials do not have any real roots. We conclude that $g$ cannot be decomposed as the sum of 6 squares of polynomials with coefficients in $\R$.
\end{proofels}

\section{Polynomials with few summands in a SOS decomposition}

In Section \ref{section:sharpbounds} we studied the maximum number of polynomials that can appear in a SOS decomposition for a strictly positive polynomial in the boundary of the SOS cone. We study now the minimal number of polynomials that can appear in such decompositions. For fixed $n$ and $d$, we are interested in finding the minimal number $s$ such that there exists a strictly positive polynomial in the boundary of $\Sigma_{n,2d}$ that can be decomposed as the sum of squares of $s$ polynomials.


 Theorems \ref{blek36} and \ref{blek44} state that any strictly positive polynomial in the boundary of $\Sigma_{n,2d}$, where $(n,2d) = (3,6)$ and $(4,4)$, is the sum of exactly $n$ squares of polynomials, and the polynomials in the decomposition are a sequence of parameters. In Section \ref{section:sharpbounds} we constructed examples containing a sequence of parameters. All these examples suggest that every sum of squares in the boundary is the sum of squares of at least $n$ polynomials. Moreover, in \cite[Corollary 2.3]{blekherman}, G. Blekherman proves that if a form $\ell$ spans an extreme ray of $\Sigma^*_{n,2d}$ spans an extreme ray, then the forms in the kernel of $Q_{\ell}$ have no common projective zeroes, real or complex, and hence contain a sequence of parameters.

Contrary to what can be expected by the examples and results above, in this section we give an example of a strictly positive polynomial $f\in\partial\Sigma_{6,4}$ that admits a SOS decomposition with less than $n = 6$ polynomials. Moreover, the polynomials appearing in the decomposition  have a common complex root.

Before giving the example, we make a useful observation.

\begin{lemma}
If $f = p_1^2 + \dots + p_r^2 \in H_{n,2d}$ is a strictly positive polynomial in the boundary of $\Sigma_{n,2d}$, and $g_1, \dots, g_s$ are homogeneous polynomials of degree $d$ in new variables $x_{n+1}, \dots, x_{n+m}$, then $h = f + g_1^2 + \dots + g_s^2$ is a strictly positive polynomial in the boundary of $\Sigma_{n+m, 2d}$.
\end{lemma}
\begin{proofels}
Let $\ell \in \Sigma_{n,d}^*$ be a linear form that vanishes in $f$, and let $Q: H_{n,d} \rightarrow \R$ be the associated quadratic form. Then
$$
\tilde Q: H_{n+m, d} \rightarrow \R, \
\tilde Q(p) = Q(p(x_1, \dots, x_n, 0, \dots, 0))
$$
is a non-trivial positive semidefinite quadratic form that vanishes in
$$\{p_1, \dots, p_r, g_1, \dots, g_s\}.$$
Hence the polynomials in any decomposition of $h$ are in the kernel of $\tilde Q$ and $h$ is not in the interior of $\Sigma_{n+m,2d}$.
\end{proofels}

We can now present  the example we were looking for.

\begin{example}[sum of 5 squares in 6 variables with common complex roots]
\label{ex:64sumOf5}
Let $p_1$, $p_2$, $p_3$, $p_4$ be as in Example \ref{ex1:54} and let $p_5 = x_5^2+x_6^2$ (instead of $p_5 = x_5^2$ as in that example). Then $f = p_1^2+p_2^2+p_3^2+p_4^2+p_5^2$ is a strictly positive polynomial, which by the above remark is in the boundary of $\Sigma_{6,4}$. The common roots of
$\{p_1, \dots, p_5\}$ are
$$\{x_1=0, x_2=0, x_3=0, x_4=0, x_5 = I x_6\},$$
where $I$ is a root of $Z^2+1$. Moreover, it can easily be verified using SEDUMI that the given decomposition is the unique sum of squares decomposition of $h$ (up to orthogonal equivalence).
\end{example}

We can easily extend this example to more variables.

\begin{example}[sum of 5 squares in $n \ge 6$ variables with common complex roots]
\label{ex:6nsumOf5}
Let $p_1$, $p_2$, $p_3$, $p_4$ be as in Example  \ref{ex1:54} and let $p_5 = x_5^2+x_6^2 + \dots + x_n^2$, $n \ge 6$. Then $f = p_1^2+p_2^2+p_3^2+p_4^2+p_5^2$ is a strictly positive polynomial in the boundary of $\Sigma_{n,4}$ and the common roots of
$\{p_1, \dots, p_5\}$ satisfy
$$\{x_1=0, x_2=0, x_3=0, x_4=0, x_5^2 + x_6^2 + \dots + x_n^2 = 0\}.$$
\end{example}

\section*{Acknowledgments}
The authors would like to thank Jose Capco, Gabriela Jeronimo, Mart\'in Mereb, Daniel Perrucci and Claus Scheiderer for fruitful discussions. Santiago Laplagne would like to thank Claus Scheiderer and all the crew of Konstanz University for their kind hospitality during his visit to Konstanz.

\bibliographystyle{amsplain}
\bibliography{rationalSOS}

\end{document}